\documentclass{article}

\usepackage{amsmath,amssymb,amsthm}

\newtheorem{definition}{Definition}[section]
\newtheorem{proposition}[definition]{Proposition}
\newtheorem{remark}[definition]{Remark}
\newtheorem{theorem}[definition]{Theorem}

\newtheorem{example}[definition]{Example}

\newtheorem{application}[definition]{Application}

\def\N{{\mathbb N}}
\def\R{{\mathbb R}}

\newcommand{\eps}{\varepsilon}

\newcommand{\ba}{\begin{array}} \newcommand{\ea}{\end{array}}
\newcommand {\bea} {\begin{eqnarray}} \newcommand {\eea} {\end {eqnarray}}
\newcommand{\beq}{\begin{equation}} \newcommand{\eeq}{\end{equation}}
\newcommand{\be}{\begin{enumerate}} \newcommand{\ee}{\end{enumerate}}
\newcommand {\bua} {\begin{eqnarray*}}
\newcommand {\eua} {\end {eqnarray*}}

\newcounter{ct}

\newcommand{\ds}{\displaystyle}

\newcommand{\ra}{\rightarrow}

\newcommand{\se}{\subseteq}

\newcommand{\remin}{\mathop{-\!\!\!\!\!\hspace*{1mm}\raisebox{0.5mm}{$
\cdot$}}\nolimits}

\let\OLDthebibliography\thebibliography
\renewcommand\thebibliography[1]{
  \OLDthebibliography{#1}
  \setlength{\parskip}{0pt}
  \setlength{\itemsep}{0pt plus 0.3ex}
}

\newcommand{\A}{S}

\begin{document}

\title{Fej\'er monotone sequences revisited}

\author{Ulrich Kohlenbach and Pedro Pinto \\ 
Department of Mathematics \\  Technische Universit\" at Darmstadt\\ 
Schlossgartenstra\ss{}e 7, 64289 Darmstadt, Germany \\ 
kohlenbach@mathematik.tu-darmstadt.de, pinto@mathematik.tu-darmstadt.de}
\maketitle 

\begin{abstract}
  In this paper we introduce a localized and relativized generalization of
  the usual concept of Fej\'er monotonicity together with uniform and  quantitative versions thereof and show that the main quantitative results obtained by the 1st author together with Nicolae and Leu\c{s}tean in 2018 and with L\'opez-Acedo and Nicolae in 2019 respectively, extend to this generalization. Our framework, in particular, covers the sequence generated by the Dykstra algorithm while the latter is not Fej\'er-monotone in the ordinary sense. This gives a theoretical explanation why under a metric regularity assumption one obtains an explicit rate of convergence for Dykstra's algorithm which was proved recently by the 2nd author.
\end{abstract}

\section{Introduction}

Let $X$ be a real Banach space, $C\subseteq X$ a nonempty subset and $(x_n)$ a sequence of points in $X$. An important feature enjoyed by many iterative methods in convex optimization is that of Fej\'er monotonicity: $(x_n)$ is Fej\'er monotone w.r.t.\ $C$ if
\[
\forall p\in C\ \forall n\in \N \left( \| x_{n+1}-p\|\leq \| x_n-p\| \right).
\]
Fej\'er monotone algorithms are frequently favored due to their good asymptotic behavior but, in general are only weakly convergent to some point in the set $C$. As most applications are naturally restricted to a finite dimensional setting, weak convergence gets then upgraded into strong convergence. Nevertheless, the general consensus appears to be that Fej\'er monotone methods are weakly convergent, and that strong convergence will prevent Fej\'er monotonicity since the sequence must eventually exhibit a preference towards some point in $C$.\\

In \cite{K-L-N,K-Lo-N}, in a metric setting, the property of Fej\'er monotonicity was extensively studied through the lenses of proof mining techniques. In \cite{K-L-N}, it was shown that under a compactness assumption one can construct a rate of metastability in the sense of Tao for $(x_n)$ if a quantitative version of Fej\'er monotonicity holds. 
In \cite{K-Lo-N}, a notion of modulus of regularity was introduced. It was shown that the availability of a modulus of regularity for a Fej\'er monotone sequence will entail the construction of a rate of convergence without any compactness assumption and using just the usual non-quantitative form of Fej\'er monotonicity. This points to the nontriviality of some metric regularity assumptions in the literature, as it is known that already for fairly simple algorithms no computable rate of convergence exists. That notwithstanding, the metric regularity assumption pertains to the specifics of the underlying set $C$ and in instances where the set $C$ enjoys nice geometry properties, a modulus of regularity (and -- correspondingly -- a rate of convergence) turns out to be available.\\

Recently~\cite{Pinto23}, the second author provided a proof-theoretical analysis of the strong convergence of Dykstra's algorithm of cyclic projections. Strikingly, even here without the sequence being Fej\'er monotone, the existence of a modulus of regularity allowed for rates of convergence. Since the iteration fails to be Fej\'er monotone, this result escapes the reach of the general results in \cite{K-Lo-N}. 

In this paper, we introduce generalized notions of {\em locally relativized Fej\'er monotonicity} and show that the main results in \cite{K-L-N,K-Lo-N} extend to 
these notions. In particular, this provides a theoretical justification for \cite[Section~4]{Pinto23}.

\section{Locally $\A$-relativized Fej\'er monotone sequences}
\label{section-approximate}

In the following, $(X,d)$ is a metric space and $F\subseteq X$ a nonempty subset. As in \cite{K-L-N}, whose notations we follow, we assume that  
\[
F=\bigcap_{k\in\N}\tilde{F}_k,
\]
where $\tilde{F}_k\se X$ for every $k\in\N$ and we say that the family  $(\tilde{F}_k)$ is a {\em representation} of $F$. It is clear that $F$ has a trivial representation, by letting $\tilde{F}_k:=F$ for all $k$ but the intended interpretation is that 
\[
AF_k:=\bigcap_{l\le k}\tilde{F}_l
\]
is some weakened approximate form of $F$. A point $p\in AF_k$ is said to be a {\em $k$-approximate $F$-point}.\\[1mm]
In the following we always view $F$ not just as a set but we suppose it is equipped with a representation $(\tilde{F}_k)$ to which we refer implicitly in many of the notations introduced below.  

Let $(x_n)$ be a sequence in $X$ and $\A(n,k)$ be an arbitrary property about $(n,k)\in\N^2$. We think of $k$ as an error $\delta>0$ via $\delta=\frac{1}{k+1}.$ Throughout this paper we assume that $\A(n,k)$ is monotone in $\delta$ in the sense of
\[
\forall n,k_1,k_2\in\N \left(k_1\le k_2\wedge \A(n,k_2)\to \A(n,k_1)\right).
\]
Note that if $\A$ does not satisfy this we may replace it by $\tilde{\A}(n,k):=\bigwedge^{k}_{i=0} \A(n,i)$.

\begin{definition}
$(x_n)$ has approximate $F/\A$-points if $\forall k\in\N \,\exists n\in \N \left(x_n\in AF_k \wedge \A(n,k)\right)$.\\
A function $\Phi:\N\to\N$ is called an approximate $F/\A$-bound if
\[
\forall k\in \N\,\exists n\le \Phi(k) \left(x_n\in AF_k\wedge \A(n,k)\right).
\]
We assume w.l.o.g.\ that $\Phi$ is nondecreasing since, otherwise, we may take\\ $\Phi^M(k):=\max\{ \Phi(i):i\le k\}$.
\end{definition}

\begin{definition}[\cite{K-L-N}] \label{explicit-closed}
We say that  $F$ is {\em explicitly closed} (w.r.t.\ the representation $(\tilde{F}_k)$) if 
\[
\forall p\in X \left( \forall 
N,M\in\N \left(AF_M\cap \overline{B}\left(p, \frac{1}{N+1}\right)\ne \emptyset\right)\ra p\in F\right).
\]
\end{definition}

As in \cite[Section~4]{K-L-N}, let $G,H:(0,\infty)\to (0,\infty)$ be functions which possess moduli $\alpha_G,\beta_H:\N\to\N$ such that 
\[
\forall k\in\N \, \forall a\in (0,\infty)
\left( a\le \frac{1}{\alpha_G(k)+1} \to G(a) \le \frac{1}{k+1}\right)  
\]
and
\[ 
 \forall k\in\N \, \forall a\in (0,\infty) 
\left(H(a)\le \frac{1}{\beta_H(k)+1} \to a \le  \frac{1}{k+1}\right).
\]

Let $(x_n)$ be a sequence in the metric space $(X,d)$ and $\emptyset \ne F\se X$.\\
The next definition generalizes the concept of $(G,H)$-Fej\'er monotonicity from \cite[Definition~4.1]{K-L-N}:

\begin{definition}\label{fejer1} 
$(x_n)$ is locally $\A$-relativized {\em $(G,H)$-Fej\'er monotone} w.r.t.\ $F$ if 
\[
(*)\left\{ \ba{l} \forall p\in F\,\forall r\in\N\,\exists m\in\N\,\forall n\in\N \\
\hspace*{5mm} \left( d(x_n,p)<\frac{1}{m+1}\wedge \A(n,m)\to 
\forall l\in\N\,(H(d(x_{n+l},p))\le G(d(x_n,p))+\frac{1}{r+1})\right). 
\ea \right.
\]
\end{definition}

\begin{remark}\rm
If $(x_n)$ is $(G,H)$-Fej\'er monotone w.r.t.\ $F$ in the sense of \cite[Definition~4.1]{K-L-N}, then it, in particular, is locally $\A$-relativized {\em $(G,H)$-Fej\'er monotone} w.r.t.\ $F$ for any property $\A(n,m)$.
\end{remark}

The next theorem generalizes \cite[Proposition~4.3]{K-L-N} to locally $\A$-relativized {\em $(G,H)$-Fej\'er monotone} sequences:

\begin{theorem}\label{general-locally-fejer} 
Let $X$ be a compact metric space and $F$ be explicitly closed. Assume that $(x_n)$ is locally $\A$-relativized $(G,H)$-Fej\'er monotone with respect to $F$ and that $(x_n)$ has approximate $F/\A$-points. Then $(x_n)$ converges to a point $x\in F$. 
\end{theorem}

\begin{proof}
For each $k\in\N$ let $n_k\in\N$ be such that
\[
x_{n_k}\in AF_k\wedge \A(n_k,k)
\]
and define $y_k:=x_{n_k}$. Since $X$ is compact, $(y_k)$ has a convergent subsequence $(y_{k_l})$. Let $x:=\lim y_{k_l}$. Similarly to \cite{K-L-N} one shows that 
\begin{enumerate}
\item[$(i)$] $x$ is an adherent point of $S_{k}:=\{ x_n:n\in\N\wedge \A(n,k)\}$ for all $k\in\N$ and
\item[$(ii)$] (using that $F$ is explicitly closed) $x\in F$.
\end{enumerate}
\textbf{Proof of $(i)$:} Let $k,m\in\N$. We want to show that
\[
\exists y\in S_k \left( d(x,y)\le\frac{1}{m+1}\right).
\]
Let $l\ge k$ be so large that $d(y_{k_l},x)\le \frac{1}{m+1}$. By construction
\[
y_{k_l}=x_{n_{k_l}}\in AF_{k_l} \wedge \A(n_{k_l},k_l).
\]
Since $(y_{k_l})$ is a subsequence of $(y_k)$ we have that $k_l\ge l\ge k$ and so by the monotonocity properties of $AF_K$ and $\A(n,K)$ in $K\in\N$
\[
y_{k_l}\in AF_k\wedge \A(n_{k_l},k).
\]
Hence we may take $y:=y_{k_l}$. \hfill$\blacksquare$\\

\noindent\textbf{Proof of $(ii)$:} By the proof of $(i)$, the premise of the property of $F$ being `explicitly closed' is satisfied for $p:=x$ and so $x\in F$. \hfill$\blacksquare$\\

\noindent We now show using $(i)$ and $(ii)$ that $\lim x_n=x$: let $r\in\N$ be arbitrary and define $r':=2\beta_H(r)+1$. Let for this $r'$ and $p:=x\in F$ be $m$ as in Definition~\ref{fejer1} and also satisfying that
\[
m\ge \max\left\{\alpha_G\left(2\beta_H(r)+1\right),2\beta_H(r)+1\right\}
\]
(this is possible to achieve by the monotonicity of $\A(n,m)$ in $m$). By $(i)$,
\[
\exists n\in\N \left( d(x_n,x)<\frac{1}{m+1}\wedge \A(n,m)\right).
\] 
Hence by Definition~\ref{fejer1} applied to this $n$:
\[
\forall l\in\N \left( H(d(x_{n+l},x))\le G(d(x_n,x))+\frac{1}{2\beta_H(r)+2}
\stackrel{\alpha_G\mbox{-def.}}{\le} \frac{1}{\beta_H(r)+1}\right)
\]
and so by $\beta_H$-definition 
\[
\forall l\in\N \left(d(x_{n+l},x)\le \frac{1}{r+1}\right).\qedhere
\]
\end{proof}

\begin{application}\label{application1} \rm
Let $C_1,\ldots,C_N\subseteq H$ be closed and convex subsets of a (real) Hilbert space with
\[
C:=\bigcap^N_{i=1} C_i\not=\emptyset
\]
and $P_i$ be the metric projection onto $C_i$ for $i=1,\ldots,N$. Let $(x_n),(q_n)$ be the sequences in $H$ generated by Dykstra's algorithm with $x_0$ as starting point. Let $z\in C$ and $\N\ni b>0$ with $\| z-x_0\|\le b$. Then, by \cite[Lemma~3.4]{Pinto23}, $(x_n)\subset \overline{B}(z,b)$. Define 
\[
f(x):=\max\limits_{i=1,\ldots,N} \| x-P_ix\|.
\] 
Note that $C= \bigcap^{\infty}_{k=0}AF_k(:=\{ x\in X : |f(x)|\le\frac{1}{k+1})\})=\text{zer}\;f$ is explicitly closed.\\ 
We also define $G(a):=H(a):=a^2$ and may take as in \cite[Lemma~7.10]{K-L-N}
\[
\alpha_G(k):=\lceil\sqrt{k}\rceil,\ \beta_H(k):=k^2.
\]
Now define (where for negative indices $k$ we take $x_k$ as arbitrary points in $H$ while $q_k:=0$) 
\[
\A(n,r):=\left(\sum^{n}_{k=n-N+1}\langle x_k-x_n,q_k\rangle <\frac{1}{r+1}\right).
\]
Since $(x_n)$ is locally $\A$-relativized $(G,H)$-Fej\'er monotone w.r.t.\ $C$ and possesses approximate $C/\A$-points (these facts are proven in stronger form in Application~\ref{application2} below), Theorem~\ref{general-locally-fejer} (applied to $X:=C\cap\overline{B}(z,b)$) gives a simple proof of the convergence of $(x_n)$ towards a point in $C$ if $H$ is finite dimensional (compared to the much more complicated strong convergence proof known in the general case, see e.g. \cite{BauCom}).
\end{application}

\begin{application}\label{new-application}\rm 
Let $(X,d)$ be compact metric space and $F$ be explicitly closed. For each $r\in\N$, consider the subset $F_r\subseteq F$ of points $p$ in $F$ which satisfy
\[
\exists m\in\N\,\forall n\in\N \left( d(x_n,p)<\frac{1}{m+1}\wedge \A(n,m)\to \forall l\in\N\,(H(d(x_{n+l},p))\le G(d(x_n,p))+\frac{1}{r+1})\right).
\]
Assume that
\[
\forall r\in \N\, \forall q\in F\setminus F_r\,\exists N \in\N \, \forall n \in \N \left( \A(n,N)\to d(x_n,q)\ge \frac{1}{N+1}\right).
\]
If $(x_n)$ has approximate $F/\A$-points, then $(x_n)$ strongly converges to a point in $F$: it suffices to show that $(x_n)$ is locally $\A$-relativized $(G,H)$-Fej\'er monotone w.r.t.\ $F$. Let $r\in\N$ and $p\in F$. If $p\in F_r$, then $(*)$ holds by definition of $F_r$. If $p\not\in F_r$, then $(*)$ holds for $m:=N$ since $\A(n,m)$ implies that $d(x_n,p)\ge \frac{1}{m+1}$.
\end{application}


By Theorem \ref{general-locally-fejer} a sequence $(x_n)$ strongly converges to a point $x\in F$ (for explicitly closed $F$) if for some condition $\A$, it is locally $\A$-relativized $(G,H)$-Fej\'er monotone and has approximate $F/\A$-points. We now show that also the converse holds (with $G:=H:=Id)$ if $F$ satisfies the following condition (which in most applications is trivially satisfied, e.g. when $F:=\text{zer}\;f$ for some continuous $f:X\to \R$ and $AF_k:=\{ x\in X: |f(x)|\le \frac{1}{k+1} \}$).

\begin{definition}\label{proper-approximation} 
We say that $F$ is properly approximated by $(AF_k)$ if 
\[
\forall x\in F\,\forall k\in\N\,\exists m\in\N\, \forall y\in X \left(d(x,y)\le\frac{1}{m+1}\to y\in AF_k\right).
\]
\end{definition}

\begin{proposition}
If $F$ is properly approximated by $(AF_k)$ and $(x_n)$ converges to $x\in F$, then for a suitable condition $\A$, the sequence $(x_n)$ is locally $\A$-relativized $(Id,Id)$-Fej\'er monotone and has approximate $F/\A$-points.
\end{proposition}

\begin{proof} Let $x\in F$ and assume that $(x_n)$ converges to $x.$
Take
\[
\A(n,r):=\forall k\ge n \left( d(x_k,x)<\frac{1}{2r+2}\right).
\]
Then clearly for all $p\in X, n,r\in\N$ 
\[
\A(n,r)\to \forall l\in\N \left( d(x_{n+l},p)\le d(x_{n+l},x_n)+d(x_n,p)\le d(x_n,p)+\frac{1}{r+1}\right)
\]
and so $(x_n)$ satisfies $(*)$ with $m:=r$. Moreover, $(x_n)$ has approximate $F/\A$-points: given $r\in\N$, let $m$ be so large that $d(x,y)\le \frac{1}{2m+2}$ implies $y\in AF_r$. Take $n\in\N$ be so large that $\A(n,\max\{ m,r\})$. Then $\A(n,r)$ and $x_n\in AF_r$.
\end{proof}

The relevance of Theorem \ref{general-locally-fejer}, of course, stems from the fact that in applications $\A$ will be such that it is much easier to construct an approximate $F/\A$-point bound than a rate of convergence for $(x_n)$ (which is needed for the specific $\A$ used in the proof above).

\section{Uniform  locally $\A$-relativized $(G,H)$-Fej\'er monotone sequences w.r.t.  $F$-points}
We now strengthen the Definition~\ref{fejer1} by demanding that `$\exists m\in\N$' is uniform w.r.t. $p\in F$.
\begin{definition}\label{fejer2}
A sequence $(x_n)$ in $X$ is uniformly locally $\A$-relativized $(G,H)$-Fej\'er monotone w.r.t. $F$ with modulus $\rho:\N\to\N$ if 
\[
\ba{l} \forall p\in F\,\forall r,n\in\N\\
\hspace*{1cm}\left( d(x_n,p)< \frac{1}{\rho(r)+1}\wedge \A(n,\rho(r))\to \forall l\in\N\ (H(d(x_{n+l},p))\le G(d(x_n,p))+ \frac{1}{r+1}\right). \ea
\]
\end{definition}

\begin{remark}\rm 
 Even with this strengthened version, any ordinary $(G,H)$-Fej\'er monotone sequence in the sense of \cite[Definition~4.1]{K-L-N} (and hence by taking $G:=H:=Id$ any Fej\'er monotone sequence) satisfies the above definition for any choice of $\A(n,r)$ and any $\rho$.
\end{remark}

Let $f:X\to\overline{\R}:=\R\cup\{+\infty\}$ an arbitrary function with zer $f:=\{ x\in X:f(x)=0\}\neq \emptyset$. We recall the following definition from \cite{K-Lo-N} (but written in $1/(k+1)$-form rather than with $\eps>0$):
\begin{definition}[{\cite[Definition~3.1]{K-Lo-N}}]
Let $z \in \text{zer}\; f$ and $b > 0$. Then $\mu :\N  \to \N$ is a \textit{modulus of regularity} for $f$ w.r.t.\ $\text{zer}\; f$ and $\overline{B}(z,b)$ if for all $k\in\N$ and $x \in \overline{B}(b,r)$ we have the following implication
\[
|f(x)| < \frac{1}{\mu(k)+1} \; \rightarrow \; \mbox{\rm dist}(x, \text{zer}\; f) < \frac{1}{k+1}.
\]
\end{definition}

Let $b>0$ and $(x_n)$ be a sequence in $\overline{B}(z,b)$ for some $z\in {\rm zer} f:=\{ x\in X : f(x)=0\}$ and $\mu$ be a modulus of regularity w.r.t. $\text{zer}\; f$ and $\overline{B}(z,b)$.\\ 
Let $\Phi:\N\to\N$ be an approximate $F/\A$-bound for $(x_n)$ with $F:=\text{zer}\; f=\bigcap^{\infty}_{k=0} AF_k$, where $AF_k:=\{ x\in X:|f(x)|\le\frac{1}{k+1})\},$ more precisely
\[
\forall k\in\N\,\exists n\le\Phi(k) \left(|f(x_n)|<\frac{1}{k+1}\wedge \A(n,k)\right).
\]
Recall that we assume $\A(n,m)$ to be monotone in $m.$

\begin{theorem}\label{theorem-fejer2}
Let $\mu$ be a modulus of regularity w.r.t. $F:=\text{zer}\;f$ and $\overline{B}(z,b)$. If $(x_n)$ is a sequence in $\overline{B}(z,b)$ which is uniformly locally $\A$-relativized $(G,H)$-Fej\'er monotone w.r.t.\ $F$ with modulus $\rho$ and if $\Phi$ is an approximate $F/\A$-bound for $(x_n)$, then $(x_n)$ is a Cauchy sequence with rate $\Psi(2k+1)$ for
\[
\Psi(k):=\Phi\left(\max\left\{ \rho(2\beta_H(k)+1),\mu\left(\max\left\{\alpha_G(2\beta_H(k)+1),\rho(2\beta_H(k)+1)\right\}\right)\right\} \right), \mbox{i.e.}
\]
\[
\forall k,m,\tilde{m}\in\N \left( m,\tilde{m}\ge \Psi(2k+1)\to \left(\| x_m-x_{\tilde{m}}\| \le\frac{1}{k+1}\right)\right)
\]
and 
\[
\forall k\in\N\,\forall n\ge\Psi(k) \left(\text{\rm dist}(x_n,\text{zer}\;f) \le \frac{1}{k+1}\right).
\]
Moreover, if $X$ is complete and $\text{zer}\;f$ is closed, then $(x_n)$ converges to a zero of $f$ with rate of convergence $\Psi(2k+1)$.
\end{theorem}

\begin{proof}
Let $k\in\N$ be given. By the definition of $\Phi$ there exists an $n\le\Psi(k)$ s.t.
\[
(1)\ \A(n,\rho(2\beta_H(k)+1))\wedge |f(x_n)|<\frac{1}{\mu(\max\{ \alpha_G
(2\beta_H(k)+1),\rho(2\beta_H(k)+1)\})+1}.
\]
Hence by $\mu$-definition
\[
(2)\ \exists p\in \text{zer}\;f \left(d(x_n,p)<\frac{1}{\max\left\{\alpha_G(2\beta_H(k)+1),\rho(2\beta_H(k)+1)\right\}+1}
\right).
\]
By the definition of $\rho$ applied to $p$, $(1)$ and $(2)$ imply that
\[
(3)\ \forall l\in\N \left( H(d(x_{n+l},p))\le G(d(x_n,p))+\frac{1}{2\beta_H(k)+2}\stackrel{(2),\alpha_G\mbox{-def}}{\le}\frac{1}{\beta_H(k)+1}\right).
\]
Hence by the definition of $\beta_H$
\[
(4)\ \forall l\in\N\,\left( d(x_{n+l},p)\le \frac{1}{k+1}\right)
\]
and so 
\[
(5)\ \forall m,\tilde{m}\ge n\,\left( d(x_m,x_{\tilde{m}})\le d(x_m,p)+d(x_{\tilde{m}},p) \le\frac{2}{k+1}\right).
\]
$(4)$ and $(5)$ establish the first two claims.\\ 
For the remaining claim we argue as in the proof of \cite[Theorem~4.1]{K-Lo-N}: let $X$ be complete and $\text{zer}\;f$ be closed. Then by the above $z':=\lim x_n$ exists and by the Cauchy rate $(5)$ we get that 
\[
(6)\ \forall m\ge\Psi(2k+1) \left( d(x_m,z')\le\frac{1}{k+1}\right).
\]
Together with $(4)$ this yields
\[
(7)\ \text{dist}(z',\text{zer}\;f)\le \text{dist}(x_m,\text{zer}\;f)+d(z',x_m) \le \frac{1}{2(k+1)}+\frac{1}{k+1}< \frac{2}{k+1}.
\]
As $k$ was arbitrary, we obtain that dist$(z',\text{zer}\;f)=0$ which implies that $z'\in\text{zer}\;f$ as $\text{zer}\;f$ is closed. 
\end{proof}

\begin{remark}\rm
Note that in Theorem \ref{theorem-fejer2}, $X$ is not required to be compact.
\end{remark}

\begin{application}\label{application2}\rm 
We now show that the rate of convergence for the Dykstra cyclic projection algorithm, recently obtained in \cite{Pinto23} under a metric regularity assumption in the case of an arbitrary real Hilbert space $H,$ is covered by Theorem \ref{theorem-fejer2} (but not by Theorem 4.1 in \cite{K-L-N} as this algorithm is not $(G,H)$-Fej\'er monotone but only uniformly locally $\A$-relativized so for a suitable $\A$): we use the notations from Application \ref{application1}. In \cite[Section~4]{Pinto23} it is assumed that $C_1,\ldots,C_N$ are metrically regular with a modulus $\mu:\N^2\to\N$ in the sense of \cite[Definition~4.6]{K-Lo-N} (we switch here from the $\eps/\delta$-formulation used in \cite{K-Lo-N} to $\frac{1}{k+1}$ instead of $\eps$ etc.\ to fit the notations used in \cite{K-L-N} and above), i.e. 
\[
\forall k,r\in\N\,\forall x\in\overline{B}(z,b)
\left(\bigwedge^N_{i=1}\| x-P_ix\|< \frac{1}{\mu_r(k)+1}\to 
\exists p\in C\, \left(\| x-p\|<\frac{1}{k+1}\right)\right).
\]
As shown in the proof of \cite[Corollary~4.8]{K-Lo-N} (and the sentence after that proof), such a $\mu_r$ is a modulus of regularity for $\text{zer}\; f$ and $\overline{B}(z,b)$, where 
\[
f(x):=\max\limits_{i=1,\ldots,N} \| x-P_ix\|.
\] 
Note that $C=\text{zer}\;f$.\\ 
We also define $G(a):=H(a):=a^2$ and may take as in \cite[Lemma~7.10]{K-L-N}
\[
\alpha_G(k):=\lceil\sqrt{k}\rceil, \ \beta_H(k):=k^2.
\]
Now define
\[
A(n,r):=\left(\sum^{n}_{k=n-N+1}\langle x_k-x_n,q_k\rangle <\frac{1}{r+1}\right).
\]
Implicitly in the proof of \cite[Theorem~4.2]{Pinto23}, an $C/\A$-bound $\Phi$ for $(x_n)$ is constructed.\footnote{In the notation of \cite{Pinto23} this bound can be taken as $\alpha(b,N,\eps,\Phi(b,N,\eps,\cdot))+\Phi(b,N,\eps,\alpha(\ldots))$ with $\eps:=1/(r+1)$.}\\
Define $\rho(r):=8b(r+1)-1$. It remains to show that $(x_n)$ is locally $\A$-relativized $(G,H)$-Fej\'er monotone (with $G,H,\A$ as above) w.r.t.\ $C$ with modulus $\rho$: let $p\in C, r\in\N$ and $n\in\N$ be such that
\[
\| x_n-p\|< \frac{1}{\rho(r)+1} \wedge \A(n,\rho(r)).
\]
By $(7)$ in \cite{Pinto23} one has for all $l\ge n$
\begin{align*}
\| x_l-p\|^2 &\le \| x_n-p\|^2+2 \underbrace{\sum^n_{k=n-N+1}\langle x_k-p,q_k
\rangle}_{\stackrel{!}{\le} 1/(2r+2)} -2\underbrace{\sum^l_{k=l-N+1} \langle x_k-p,q_k
\rangle}_{\ge 0,\ \text{by \cite[\rm L.3.1(iii)]{Pinto23}}}\\ 
&\le \| x_n-p\|^2+\frac{1}{r+1}.
\end{align*}
Here `!' follows from $\| x_n-p\|<\frac{1}{8b(r+1)}$ and $\sum\limits^n_{k=n-N+1}\langle x_k-x_n,q_k\rangle <\frac{1}{4(r+1)}$ by reasoning as in \cite[p.17]{Pinto23}:
\begin{align*}
 \sum\limits^n_{k=n-N+1}\langle x_k-p,q_k\rangle&= \sum\limits^n_{k=n-N+1}\langle x_k-x_n,q_k\rangle +
\underbrace{\langle x_n-p,\sum\limits^n_{k=n-N+1} q_k \rangle}_{=\langle x_n-p,x_0-x_n\rangle,\ \text{by \cite[\rm L.3.1(ii)]{Pinto23}}}\\ 
&\le \sum\limits^n_{k=n-N+1}\langle x_k-x_n,q_k\rangle + \| x_n-p\|\cdot\| x_0-x_n\|\\
&\le
\frac{1}{4(r+1)}+2b\cdot \frac{1}{8b(r+1)} =\frac{1}{2(r+1)}.
\end{align*}

\end{application}

\mbox{ } 
 
We now discuss a simple argument which in some sense gives a reversal to Theorem~\ref{theorem-fejer2}. Fix a function $I:\N\times X\to X$ standing for some iterative process employed to approximate zeros of $f$. Consider the following natural condition on $I$ (see remark~\ref{remark-I} below):
\begin{equation}\label{hypothesis-I}\tag{$+$}
	\forall n\in\N\ \forall p\in \text{zer}\;f \left( I(n,p)=p \right),
\end{equation}
stating that the iterative process leaves the points in $\text{zer}\;f$ unchanged (such iterations are called rectractive in \cite{Neumann}). Consider a sequence $(x_n)$ recursively defined by the function $I$, i.e.\ $x_0\in X$ and $x_{n+1}:=I(n,x_n)$. Therefore, the assumption \eqref{hypothesis-I} on $I$ entails,
\begin{equation}\label{hypothesis-II}\tag{$++$}
	\forall x_0\in X \left( x_0\in\text{zer}\;f \to \forall n\in\N \left(x_n=x_0\right) \right).
\end{equation}

\begin{definition}
We say that a function $\tau:\N^2\to\N$ is a modulus of coincidence for $I$ if it satisfies
\[
\forall n\in\N\, \forall k\in\N\, \forall x_0\in X \left( |f(x_0)|<\frac{1}{\tau(k,n)+1}\to d(x_n,x_0)< \frac{1}{k+1} \right),
\]
where $(x_n)$ is initiated at $x_0$ and recursively defined by $I$.
\end{definition}

\begin{remark}\label{remark-I} \rm 
Most iterative methods in optimization satisfy the condition~\eqref{hypothesis-I}. For example, in a normed linear space with $f(x)=\|x-T(x)\|$ for a nonexpansive map $T$, if $I(n,x)$ is defined as
\begin{enumerate}
	\item $T(x)$ -- Picard iteration,
	\item $\alpha_nx+(1-\alpha_n)T(x)$, for $(\alpha_n)\subseteq [0,1]$ -- Krasnoselski-Mann iteration,
	\item $\alpha_nx+(1-\alpha_n)T(\beta_nx+(1-\beta_n)T(x))$, for $(\alpha_n), (\beta_n) \subseteq [0,1]$ -- Ishikawa iteration,
\end{enumerate}
then the condition \eqref{hypothesis-I} is satisfied. Another well-known iterative method is the Halpern iteration, $I(n,x)=\alpha_nu+(1-\alpha_n)T(x)$. Despite not satisfying \eqref{hypothesis-I}, it still satisfies \eqref{hypothesis-II} whenever $u=x_0$. One easily checks that for the Ishikawa iteration $\tau(k,n):=2n(k+1)$ gives a modulus of coincidence, and for all the other cases (including the Halpern iteration) one can take $\tau(k,n):=n(k+1)$. Naturally, if one imposes further conditions on the parameter sequences $(\alpha_n)$, $(\beta_n)$, then the definition of the function $\tau$ may be improved.
\end{remark}

We have the following easy result.
\begin{proposition}\label{reversal}
Consider a function $I$ as above, and let $\tau$ be a modulus of coincidence for $I$. Given $z\in\text{zer}\;f$ and $b>0$, assume that for any $x_0\in \overline{B}(z,b)$ the sequence initiated at $x_0$ and recursively defined by $I$ converges to a point $\ell_I(x_0)\in\text{zer}\,f$ with a common rate of convergence $\Psi:\N\to\N$, i.e.
\[
d(x_0,z)\leq b \to \forall k\in\N\, \forall n\geq \Psi(k) \left( d(x_n,\ell_I(x_0))\leq \frac{1}{k+1}\right).
\]
Then, the function $\mu(k):=\tau(2k+1,\Psi(2k+1))$ is a modulus of regularity for $f$ w.r.t.\ $\text{zer}\;f$ and $\overline{B}(z,b)$.
\end{proposition}

\begin{proof}
Let $k\in\N$ and $x\in\overline{B}(z,b)$ be given, and assume that $|f(x)|< \frac{1}{\mu(k)+1}$. Consider $(x_n)$ to be the sequence recursively defined by $I$ with initial point $x_0=x$. By the hypothesis of $\Psi$, we have $d(x_{\Psi(2k+1)},\ell_I(x))\leq \frac{1}{2(k+1)}$. On the other hand, the assumption on $\tau$ entails $d(x_{\Psi(2k+1)},x)< \frac{1}{2(k+1)}$. We conclude that
\[
\text{\rm dist}(x,\text{zer}\,f)\leq d(x,\ell_I(x))\leq d(x,x_{\Psi(2k+1)}) + d(x_{\Psi(2k+1)},\ell_I(x))< \frac{1}{k+1},
\]
and hence $\mu$ is a modulus of regularity for $f$ w.r.t.\ $\text{zer}\,f$ and $\overline{B}(z,b)$.
\end{proof}

\begin{remark}\rm 
Theorem~\ref{theorem-fejer2} states that the assumption of a modulus of regularity in the case of a sequence which is uniformly locally $\A$-relativized $(G,H)$-Fej\'er monotone w.r.t. $F=\text{zer}\;f$, and has approximate $F/\A$-points, will converge (when $X$ is complete and $\text{zer}\;f$ is closed) to a point in $\text{zer}\;f$ with a uniform rate of convergence. Proposition~\ref{reversal} above states that, in most iterative methods, if such uniform rate of convergence is to exist, then the assumption of metric regularity is indeed necessary. This argument was used in \cite[Proposition~4.4]{K-Lo-N} in the context of the Picard iteration. It was also used recently in \cite[Proposition~4.7]{Pinto23}, in the context of Dykstra's method for solving the convex feasibility problem. In this iterative method, the computation of the function $\tau$ is more convoluted and was given in \cite[Propositions~4.5 and 4.6]{Pinto23}.
\end{remark}

\section{Uniform  locally $\A$-relativized $(G,H)$-Fej\'er monotone \\ sequences w.r.t. approximate $F$-points}

We now generalize the concept of `uniform $(G,H)$-Fej\'er monotonicity' introduced in \cite[Definition~4.6]{K-L-N}:

\begin{definition} \label{fejer3}
We say that $(x_n)$ is {\em uniformly locally $\A$-relativized $(G,H)$-Fej\'er 
monotone} w.r.t. approximate $F$-points if for all $r, n,m\in\N$ with moduli $\chi:\N^3\to\N, \rho:\N\to\N$ if   
\[
\ba{l}
\forall r,n,m\in\N\,\forall p\in X \\
\left( p\in AF_{\chi(n,m,r)}\wedge d(x_n,p)<\frac{1}{\rho(r)+1}\wedge \A(n,\rho(r))\to \forall l\le m \left(H(d(x_{n+l},p))\le G(d(x_n,p)) + \frac1{r+1} \right)\right)\!.
\ea 
\]
\end{definition}

\begin{remark} \rm 
If is uniformly $(G,H)$-Fej\'er monotone w.r.t. $F$ with modulus $\chi$ in the sense of \cite[Definition~4.6]{K-L-N}, then $(x_n)$ is uniformly locally $\A$-relativized $(G,H)$-Fej\'er monotone w.r.t.\ approximate $F$-points with moduli $\chi$ and an arbitrary $\rho:\N\to\N$ for any property $\A$.  
\end{remark}

\begin{example}\label{application3} \rm
In the setting of Application~\ref{application1} for the case of Dykstra's algorithm, we can obtain moduli $\chi$, $\rho$ in the sense of Definition~\ref{fejer3}. Indeed, for each $r,n,m\in\N$ define
\[
\rho(r):=4(2b+1)(r+1)-1\ \text{and}\ \chi(n,m,r):=8b(n+m)(r+1) \remin 1,
\]
where $\N\ni b\geq \|z-x_0\|$ for some $z\in C$, as before. Assume that $\|x_n-p\|\leq \frac{1}{\rho(r)+1}$,
\[
p\in AF_{\chi(n,m,r)},\ \text{i.e.}\ \max_{i=1,\dots,N}\|p-P_i(p)\|\leq \frac{1}{\chi(n,m,r)+1},\ \text{and}\ \sum_{k=n-N+1}^n \langle x_k-x_n, q_k\rangle\leq \frac{1}{\rho(r)+1}
\]
for given $r,m,n\in \N$ and $p\in X$.\footnote{Recall that the last conjunct stands for the property $\A(n,\rho(r))$.} Then, similarly to the proof of \cite[Theorem~3.11]{Pinto23}, we have for $G(a)=H(a)=a^2$ and any $\ell\leq m$,
\[
H(\|x_{n+\ell}-p\|)\leq G(\|x_n-p\|)+2 \bigg( \underbrace{\sum_{k=n-N+1}^{n} \langle x_k-p,q_k\rangle}_{=:t_1} + \underbrace{\sum_{k=n+\ell-N+1}^{n+\ell} \langle p -x_k, q_k\rangle}_{=:t_2} \bigg).
\]
Then,
\begin{align*}
t_1&=\sum\limits_{k=n-N+1}^n \langle x_k-x_n, q_k\rangle + \sum\limits_{k=n-N+1}^n \langle x_n-p,q_k\rangle\\[1mm]
&\hspace{-9mm}\stackrel{\text{\cite[\rm L.3.1(ii),L.3.4]{Pinto23}}}{\leq} \frac{1}{\rho(r)+1} + 2b\|x_n-p\|\leq \frac{2b+1}{\rho(r)+1}=\frac{1}{4(r+1)}
\end{align*}
and (with $P_k(p)$ being an arbitrary point for negative $k$)
\begin{align*}
t_2&=\sum\limits_{k=n+\ell-N+1}^{n+\ell} \langle p-P_k(p), q_k\rangle + \sum\limits_{k=n+\ell-N+1}^{n+\ell} \underbrace{\langle P_k(p)-x_k, q_k\rangle}_{\leq 0,\ \text{by \cite[\rm L.3.1(iii)]{Pinto23}}}\\[1mm]
&\leq \sum\limits_{k=n+\ell-N+1}^{n+\ell} \|p-P_k(p)\|\cdot\|q_k\|
\stackrel{\text{\cite[\rm L.3.2]{Pinto23}}}{\leq} \frac{1}{\chi(n,m,r)+1}\sum\limits_{k=0}^{n+\ell-1}\|x_k-x_{k+1}\|\\[4mm]
&\leq \frac{2b(n+\ell)}{\chi(n,m,r)+1}\leq \frac{1}{4(\ell+1)}.
\end{align*}
Therefore $2(t_1+t_2)\leq \frac{1}{r+1}$, which concludes the proof.
\end{example}

The next theorem generalizes \cite[Theorem~5.1]{K-L-N} to sequences which are uniformly locally $\A$-relativized $(G,H)$-Fej\'er monotone w.r.t.\ approximate $F$-points:
\begin{theorem}\label{theorem-fejer3}
Let $(X,d)$ be a totally bounded metric space and $\gamma$ be a modulus of total boundedness for $X$ in the sense of {\rm \cite[Definition~2.2]{K-L-N}}. Assume that
\be
\item $(x_n)$ is uniformly locally $\A$-relativized $(G,H)$-Fej\'er monotone w.r.t. approximate $F$-points,  with moduli $\chi,\rho$;
\item $(x_n)$ has approximate $F/\A$-points, with $\Phi$ being a nondecreasing approximate $F/\A$-point bound.
\ee
Then $(x_n)$ is Cauchy and, moreover, for all $k\in\N$ and all $g:\N\to\N$
\[
\exists N\le \Psi(k,g,\Phi,\chi,\rho,\alpha_G,\beta_H,\gamma) \,\forall i,j\in [N,N+g(N)] 
\left(d(x_i,x_j)\le \frac1{k+1}\right),
\]
where $\Psi(k,g,\Phi,\chi,\rho,\alpha_G,\beta_H,\gamma):=\Psi_0(P-1):=\Psi_0(P-1,k,g,\Phi,\chi,\rho,\beta_H)$, with\\[1mm] 
$P:=\gamma\left( \max\{ \alpha_G\left(2\beta_H(2k+1)+1\right),\rho(2\beta_H(2k+1)+1)\}\right)$ 
and
\[
\ba{l}
\Psi_0(0):=\Phi(\rho(2\beta_H(2k+1)+1)),\\[1mm] 
\Psi_0(n+1):=\Phi\left(\max\{ \chi^M_g\left(\Psi_0(n),2\beta_H(2k+1)+1\right),\rho(2\beta_H(2k+1)+1)\}\right)
\ea
\]
with
\[
\chi_g(n,k):=\chi(n,g(n),k),\quad  \chi^M_g(n,k):=\max\{ \chi_g(i,k) \mid i\le n\}.
\]
\end{theorem} 

\begin{proof}
We follow closely the proof of \cite[Theorem~5.1]{K-L-N} with some decisive changes though. Let $k\in\N$ and $g:\N\to\N$. Define
\[
\varphi(k):=\min m\, \left(x_m\in AF_k \wedge \A(m,k)\right).
\] 
Both $\varphi$ and $\Phi$ are nondecreasing and $\Phi$ is a pointwise bound on $\varphi$. By induction we readily prove that $\Psi_0(n,k,g,\varphi,\chi,\rho,\beta_H)\leq \Psi_0(n+1,k,g,\varphi,\chi,\rho,\beta_H)$,  $\Psi_0(n,k,g,\Phi,\chi,\rho,\beta_H)\leq \Psi_0(n+1,k,g,\Phi,\chi,\rho,\beta_H)$ and
$\Psi_0(n,k,g,\varphi,\chi,\rho,\beta_H) \leq \Psi_0(n,k,g,\Phi,\chi,\rho,\beta_H)$ for all $n\in\N$. 

Define for every $i\in\N$ 
\[
n_i:=\Psi_0(i,k,g,\varphi,\chi,\rho,\beta_H).
\]

{\bf Claim 1:} For all $j\ge 1$ and all $0\le i<j$,  $x_{n_j}$ is a $\chi_g(n_i,2\beta_H(2k+1)+1)$-approximate $F$-point.\\

{\bf Proof of claim:}  As $j\ge 1$ and
\begin{align*} 
n_j&= \Psi_0(j,k,g,\varphi,\chi,\rho,\beta_H)\ge \varphi\left(\chi^M_g\left(\Psi_0(j-1,k,g,\varphi,\chi,\rho,\beta_H),2\beta_H(2k+1)+1\right)\right)\\
&= \varphi\left(\chi^M_g\left(n_{j-1},2\beta_H(2k+1)+1\right)\right),
\end{align*}
$x_{n_j}$ is a $\chi^M_g(n_{j-1}, 2\beta_H(2k+1)+1)$-approximate $F$-point. Since $0\le i\le j-1$, we have that $n_i\leq n_{j-1}$. Apply now the fact that $\chi_g^M$ is nondecreasing in the first argument to get that
\begin{align*}
\chi_g(n_i, 2\beta_H(2k+1)+1) 
&\le  \chi^M_g(n_i, 2\beta_H(2k+1)+1)\\
&\le  \chi^M_g(n_{j-1}, 2\beta_H(2k+1)+1)
\end{align*}
which establishes Claim 1. \hfill $\blacksquare$\\

\noindent{\bf Claim 2:} There exist  $\ds 0\leq I< J \le P$ satisfying 
\[
\forall l\in [n_I,n_I+g(n_I)] \left( d(x_l,x_{n_J})\le \frac1{2k+2}\right).
\]
{\bf Proof of claim:} Utilizing that $\gamma$ is modulus of total boundedness for $X$ we get that there exist  $\ds 0\leq I< J \le P$ with 
\[
(1)\ d(x_{n_I},x_{n_J})\le \min\left\{ \frac{1}{\alpha_G(2\beta_H(2k+1)+1)+1},
\frac{1}{\rho(2\beta_H(2k+1)+1)+1} \right\}.
\] 
Since $\alpha_G$ is a $G$-modulus this, in particular, implies  
\[
(2)\ G(d(x_{n_I},x_{n_J}))\le  \frac1{2\beta_H(2k+1)+2}.
\]
By the first claim, we have that $x_{n_J}$ is a $\ds \chi_g(n_I,2\beta_H(2k+1)+1)$-approximate $F$-point. By construction, $n_I$ satisfies $\A(n_I,\rho(2\beta_H(2k+1)+1))$. Applying now the uniform locally $\A$-relativized $(G,H)$-F\'{e}jer monotonicity of $(x_n)$ w.r.t.\ approximate $F$-points with $\ds r:=2\beta_H(2k+1)+1, n:=n_I, m:=g(n_I)$ and  $p:=x_{n_J}$, we get -- 
using $(1)$ again -- that for all $l\le g(n_I)$,
\[
H(d(x_{n_I+l},x_{n_J})) \le G(d(x_{n_I},x_{n_J}))+\frac1{2\beta_H(2k+1)+2} \stackrel{(2)}{\le} \frac1{\beta_H(2k+1)+1}.
\]
Since $\beta_H$ is an $H$-modulus, 
\[
\forall l\le g(n_I) \left( d(x_{n_I+l},x_{n_J})\le \frac1{2k+2} \right).
\]
which establishes Claim 2. \hfill $\blacksquare$\\

\noindent From Claim 2 it is immediate that 
\[
\forall k,l\in [n_I,n_I+g(n_I)] \left( d(x_k,x_l)\le \frac1{k+1}\right).
\]
As $n_I=\Psi_0(I,k,g,\varphi,\chi,\rho,\beta_H)\leq \Psi_0(I,k,g,\Phi,\chi,\rho,\beta_H)$ and $I\leq P-1$, we get that
\[
n_I\le \Psi_0(P-1,k,g,\Phi,\chi,\rho,\beta_H)=\Psi(k,g,\Phi,\chi,\rho,\alpha_G,\beta_H,\gamma).
\]
The theorem holds with $N:=n_I$.
\end{proof} 

{\bf Corollary to the proof:} One of the numbers $n_0,\ldots,n_{P-1}$ is a point of metastability. 

\begin{application}\rm
As discussed already at the end of Application~\ref{application1}, this relativized notion of Fej\'er monotonicity allows for a simple convergence proof of Dykstra's algorithm in the finite dimensional case. In light of the moduli $\chi$ and $\rho$ in Example~\ref{application3} (and also the $C/\A$-bound $\Phi$ implicit in \cite[Theorem~4.2]{Pinto23}, cf.\ footnote 1), by Theorem~\ref{theorem-fejer3} we moreover obtain rates of metastability for Dykstra's iteration.
\end{application}


\noindent
{\bf Acknowledgments:} \\[1mm] 
The authors were supported by the German Science Foundation (DFG 
Project KO 1737/6-2).

\end{document}